\documentclass[11pt,a4paper,reqno]{amsart}
\usepackage[utf8]{inputenc}
\usepackage[T1]{fontenc}
\usepackage{amsmath,amssymb,amsthm}
\usepackage{mathtools}
\usepackage{hyperref}
\usepackage{enumerate}
\usepackage{mathrsfs}
\usepackage{bbm}

\newtheorem{theorem}{Theorem}[section]
\newtheorem{lemma}[theorem]{Lemma}
\newtheorem{proposition}[theorem]{Proposition}

\theoremstyle{definition}

\newtheorem{remark}[theorem]{Remark}

\newcommand{\T}{\mathbb{T}}
\newcommand{\R}{\mathbb{R}}

\newcommand{\N}{\mathcal{N}}

\newcommand{\PP}{\mathbb{P}}

\begin{document}
\title[2D anisotropic Navier-Stokes equations on the torus]{Uniform large deviations and long-time dynamics for
       the 2D anisotropic Navier-Stokes equations on the torus}
\author{Chengfeng Sun$^{1}$}
\author{Hao Yu$^{1}$}
\author{Hui Liu$^{2}$}\thanks{Corresponding author, Hui Liu (ss\_liuhui@ujn.edu.cn)}
\maketitle
         $^{1}$School of Applied Mathematics,
         Nanjing University of Finance and Economics,
         Nanjing 210023, PR China

         $^{2}$School of Mathematical Sciences,
         University of Jinan,
         Jinan 250022, PR China
\begin{abstract}
  We study the two-dimensional stochastic Navier-Stokes equations on the torus
  with horizontal dissipation and additive noise.
  First, we prove a uniform large deviation principle for the solution
  paths in the energy space $C([0,T];H).$  The proof combines a large deviation principle for a linear
  Ornstein-Uhlenbeck process, a new Lipschitz estimate for the nonlinear map, and a contraction
  principle that avoids any exponential tightness condition. Second, we analyse the long-time dynamics.  By exhibiting a deterministic
  steady state immune to the horizontal dissipation, we show that exponential
  mixing cannot hold.
  Finally, we investigate the invariant measures of the system.  In the deterministic case,
  the infinite-dimensional kernel of the
  horizontal Laplacian contains infinitely many Dirac invariant measures.
  Under a suitably degenerate noise, however, the kernel supports no invariant
  measure at all; moreover, any invariant measure of this stochastic
  system, should it exist, must be orthogonal to this kernel.
  This sharply distinguishes the anisotropic equations on the torus both from the fully
  dissipative Navier-Stokes equations and from their counterparts on domains
  with no-slip boundaries.
\end{abstract}

\vspace{1em}

\section{Introduction}
\label{sec:intro}

\subsection{The anisotropic Navier-Stokes model and its motivation}
We study the two-dimensional incompressible Navier-Stokes equations on the
torus $\mathbb{T}^2 = \mathbb{R}/2\pi\mathbb{Z} \times \mathbb{R}/2\pi\mathbb{Z}$ with horizontal dissipation only:
\begin{equation}\label{eq:NSintro}
  \partial_t u + (u\cdot\nabla)u - \partial_1^2 u + \nabla p = 0,\qquad
  \nabla\cdot u = 0,\qquad x\in\mathbb{T}^2 .
\end{equation}
When the fluid is subject to a small additive random perturbation of intensity
$\varepsilon>0$, the dynamics read
\begin{equation}\label{eq:SNSintro}
  du^\varepsilon + (u^\varepsilon\cdot\nabla)u^\varepsilon \,dt + \nabla p^\varepsilon \,dt
  - \partial_1^2 u^\varepsilon \,dt = \sqrt\varepsilon \, Q^{1/2} dW,\qquad
  u^\varepsilon(0)=x\in\widetilde H^{0,1}.
\end{equation}
Here $W$ is an $H$-cylindrical Wiener process on a filtered probability space,
$H=\{u\in L^2(\mathbb{T}^2;\mathbb{R}^2):\nabla\cdot u=0,\ \int_{\mathbb{T}^2}u=0\}$, and the noise
covariance $Q^{1/2}:H\to\widetilde H^{1,1}$ is a Hilbert-Schmidt operator
(to ensure sufficient spatial regularity of the stochastic forcing).  The
operator $A=-\mathbb P\partial_1^2$, where $\mathbb P$ is the Leray projector,
acts only in the horizontal direction, vertical dissipation is entirely absent.
Models with anisotropic viscosity appear naturally in geophysical fluid dynamics,
where turbulent mixing in the horizontal direction is often orders of magnitude
stronger than in the vertical one (see, e.g., \cite{CushmanRoisin, Pedlosky}).
The global well-posedness of system (\ref{eq:NSintro}) in
the anisotropic Sobolev space $\widetilde H^{0,1}$ was proved by Liang, Zhang
and Zhu~\cite{LZZ} by exploiting a delicate balance between the divergence-free condition
 and the anisotropic energy estimates.  More recently, Cao and Guo~\cite{CaoGuo2026} further
investigated this deterministic model and provided refined regularity criteria,
thereby complementing the global existence result of~\cite{LZZ}.

\subsection{Previous results on long-time behaviour and large deviations}

In a deterministic case, Dong, Wu, Xu and Zhu~\cite{DongWuXuZhu2021}
studied the anisotropic Navier-Stokes equations on the partially periodic
domain $\mathbb{T}\times\mathbb{R}$.  By splitting the velocity into a
horizontal average and an oscillatory part, and using a Poincar\'{e}-type
inequality that holds for the oscillatory component,
they proved that the $H^1$-norm of the oscillatory
part decays exponentially for small initial data. Their method relies on the specific geometry $\mathbb{T}\times\mathbb{R}$
where the oscillatory part has zero horizontal mean.
On the fully periodic torus $\mathbb{T}^2$ considered in the present paper,
this fundamental property fails.  As a result, not only does the
oscillatory part fail to decay, but we prove that no exponential
mixing is possible at all.

When a stochastic forcing is added,  on a Dirichlet (no-slip) boundary condition,
Sun, Qiu and Tang~\cite{SPL} proved that the Markov semigroup generated by a
similar anisotropic Navier-Stokes equation with additive noise admits a
\emph{unique} ergodic invariant measure.  Their argument combines sequential
weak Feller properties with an asymptotic coupling technique, and it crucially
relies on the anisotropic Poincar\'{e} inequality
$\|u\|_{L^2}\le C\|\partial_1 u\|_{L^2}$ which is valid on bounded domains with
Dirichlet boundaries (see Lemma~2.1 of~\cite{SPL}).  On the torus,
however, this inequality fails because purely vertical shear modes
$u=(f(x_2),0)$ have zero horizontal derivative but non-zero $L^2$-norm, they
constitute the infinite-dimensional kernel $\mathcal N$ of $A$.  Consequently,
the methods of~\cite{SPL} do not carry over to the periodic setting.
Very recently, Liang~\cite{Liang2026} proved uniqueness of invariant measures
for a stochastically \emph{damped} anisotropic Navier-Stokes system on
$\mathbb{R}^2$, relying on the linear damping to compensate for the absence of
a Poincar\'{e} inequality.  Since our torus model has no such damping, his
approach does not apply to our setting.

In the context of large deviations, the classical theory for SPDEs has been
developed around the weak convergence method~\cite{BD,BDM}. In
Cerrai and Paskal~\cite{CP}, by a uniform contraction principle that requires \emph{exponential tightness}
of the linear Ornstein-Uhlenbeck process in an auxiliary space, they  proved a large deviation principle for
the invariant measures of the 2D Navier-Stokes equations with full dissipation.
For the anisotropic Navier-Stokes equations with multiplicative noise, because of the lack of vertical dissipation
and hence the exponential tightness,
Chen and Zhu~\cite{ChenZhu} only established a (non-uniform) LDP in the weak topology space
$C([0,T];H^{-1})$.  Recently, Wang~\cite{Wang2023} introduced a refined uniform contraction principle
that \emph{does not require exponential tightness}.  The price to pay is that the
linear process must satisfy an LDP in the \emph{same} auxiliary space where the
nonlinear map is locally Lipschitz.  This is exactly the route we follow in the
present paper.

\subsection{Main contributions}
Our first result is a \emph{uniform} large deviation principle
for the solution of \eqref{eq:SNSintro} in the energy space $C([0,T];H)$,
uniformly over bounded or compact sets of initial data (Theorem~\ref{th:FW} and \ref{th:mainLDP}).  To achieve
this we prove:
\begin{itemize}
  \item an LDP for the linear Ornstein-Uhlenbeck process $z^\varepsilon$ in the
    auxiliary space $G = C([0,T];H)\cap L^2(0,T;\widetilde H^{1,1})$;
  \item a crucial uniform Lipschitz estimate (Proposition~\ref{prop:Lip}) for the
    nonlinear map $\mathcal F_x : G \to C([0,T];H)$ that sends the linear orbit
    to the solution of the nonlinear stochastic equation.
    The estimate relies on a new algebraic manipulation of the anisotropic Gagliardo-Nirenberg
    inequalities and the divergence-free condition.
  \item the application of Wang's uniform contraction principle, which avoids the
    need for exponential tightness.
\end{itemize}

To our knowledge, this is the first uniform LDP in the energy space for the anisotropic Navier-Stokes equations.
It requires only a smoothing covariance $Q$ with no extra damping or compactness.

We then turn to the question of long-time asymptotics.  Exponential mixing provides a convenient route
from pathwise LDPs to invariant measures.
Our second main contribution is a rigorous proof
that \emph{exponential mixing cannot hold} for the anisotropic Navier-Stokes
equations on the torus (Theorem~\ref{th:noExpMix}).  The proof exhibits a
deterministic steady state $u_0=(\sin x_2,0)$ whose $L^2$ norm never decays.
For the stochastic equation with sufficiently small noise,
this undamped mode prevents the difference of two solutions
from decaying uniformly exponentially.
Consequently, we cannot apply the stationary solution method in~\cite{LiuTangZhang}
for proving large deviation principles
of invariant measures to our model.
That method requires exponential decay of solution differences (Lemma~5.4), which our system lacks.

In fact, whether an invariant measure exists at all remains unknown. Finally, we analyse the structure of invariant measures.
In the deterministic
limit ($\varepsilon=0$), the kernel $\mathcal N$ is infinite-dimensional.
Its elements are all steady states, so they yield infinitely many Dirac invariant measures.
This sharply contrasts with the fully dissipative case, where only $\delta_0$ is invariant.
For the stochastic system with a degenerate additive noise and its range lies in $\mathcal N$,
we prove orthogonality of any invariant measure to $\mathcal N$ (Theorem~\ref{th:degen}).
Roughly speaking, the undamped vertical shear modes cannot contribute to a statistical steady state.
We further demonstrate that under this degenerate noise, the kernel
itself supports no invariant measure at all.

\subsection{Outline of the paper}
The paper is organised as follows.  Section~\ref{sec:prelim} introduces function
spaces, anisotropic inequalities and the key  properties of the
horizontal Stokes operator.  Section~\ref{sec:linear} proves the LDP for the
linear Ornstein-Uhlenbeck process in $G$.  Section~\ref{sec:nonlinear} constructs the nonlinear map,
gives the Lipschitz estimate,
and uses the contraction principle to prove the uniform pathwise LDP.
Section~\ref{sec:mixing} is devoted to the failure of exponential mixing.
Section~\ref{sec:invmeas} examines invariant measures:
deterministic comparison, a degenerate noise example,
and necessary conditions for the stochastic case.
 Section~\ref{sec:conclusion} summarises our main results, places them
within a broader conceptual framework for partially dissipative systems,
and outlines possible directions for future work.

\section{Preliminaries}
\label{sec:prelim}

\subsection{Function spaces}
\label{sec:function-spaces}

We first recall some function spaces on the two dimensional torus $\mathbb{T}^2$.
Let $\mathbb{T}^2 = \mathbb{R}/2\pi\mathbb{Z} \times \mathbb{R}/2\pi\mathbb{Z}$ be the two-dimensional torus.
For $s,s'\ge 0,$ the anisotropic Sobolev space $H^{s,s'}(\mathbb{T}^2)$ is defined as the completion of smooth functions with respect to the norm
\[
  \|u\|_{H^{s,s'}}^2 = \sum_{k\in\mathbb{Z}^2} (1+|k_1|^2)^s(1+|k_2|^2)^{s'}|\hat u_k|^2,
\]
where $\hat u_k$ are the Fourier coefficients of $u$.
The divergence-free Sobolev spaces are
\[
  \widetilde H^{s,s'} = \bigl\{ u\in H^{s,s'}(\mathbb{T}^2;\mathbb{R}^2) : \nabla\cdot u = 0 \bigr\}.
\]

For the analysis of the Navier-Stokes equations on the torus it is customary to work in the space of divergence-free vector fields with vanishing mean value.  We therefore set
\begin{equation}\label{eq:Hdef}
  H = \Bigl\{ u\in L^2(\mathbb{T}^2;\mathbb{R}^2) : \nabla\cdot u = 0,\; \int_{\mathbb{T}^2} u(x)\,dx = 0 \Bigr\}.
\end{equation}

All divergence-free Sobolev spaces that appear in the sequel inherit the zero-mean condition, i.e.\ we always understand
\[
  \widetilde H^{s,s'} = H^{s,s'}(\mathbb{T}^2;\mathbb{R}^2) \cap H.
\]

For a vector field $u$ we write $\|u\|$ for the $L^2$-norm and $\langle u,v\rangle$ for the $L^2$ inner product.
The norm in $\widetilde H^{0,1}$ is equivalent to
\[
  \|u\|_{\widetilde H^{0,1}}^2 = \|u\|^2 + \|\partial_2 u\|^2,
\]
and in $\widetilde H^{1,1}$ to
\[
  \|u\|_{\widetilde H^{1,1}}^2 = \|u\|^2 + \|\partial_1 u\|^2 + \|\partial_2 u\|^2 + \|\partial_1\partial_2 u\|^2 .
\]

The linear Stokes operator with horizontal dissipation is defined by
\[
  A = -\mathbb{P}\partial_1^2, \qquad D(A) = \widetilde H^{2,0},
\]
where $\PP$ is the Leray projector onto $H$. In the paper we use the mixed anisotropic Lebesgue spaces $L^p_h L^q_v$ (horizontal $L^p$, vertical $L^q$) with the norm
\[
  \|u\|_{L^p_h L^q_v} =
  \Bigl( \int_{\mathbb{T}_h} \bigl( \int_{\mathbb{T}_v} |u(x_1,x_2)|^q \,dx_2 \bigr)^{p/q} dx_1 \Bigr)^{1/p}.
\]
The following anisotropic Gagliardo--Nirenberg inequalities from \cite{LZZ} are crucial.
\begin{lemma}\label{lem:GN}
  There exists \(C>0\) such that for all smooth \(f:\T^2\to\R\),
  \begin{align*}
    \|f\|_{L^2_v L^\infty_h}^2 &\le C\bigl(\|f\|_{L^2}\|\partial_1 f\|_{L^2} + \|f\|_{L^2}^2\bigr),\\
    \|f\|_{L^2_h L^\infty_v}^2 &\le C\bigl(\|f\|_{L^2}\|\partial_2 f\|_{L^2} + \|f\|_{L^2}^2\bigr).
  \end{align*}
\end{lemma}

\subsection{Kernel of the horizontal Stokes operator}

We now characterize the kernel of $A$.

\begin{lemma}\label{lem:kernel}
  Define $\N = \ker A$.  Then
  \begin{enumerate}
    \item $\N = \{ u\in H : \partial_1 u = 0 \}$;
    \item $\N$ consists exactly of the vector fields
      \begin{equation}\label{eq:kernel_form}
        u(x_1,x_2) = \begin{pmatrix} f(x_2) \\ 0 \end{pmatrix},
        \qquad f\in L^2(\mathbb{T}),\quad \int_{\mathbb{T}} f(x_2)\,dx_2 = 0;
      \end{equation}
    \item $B(u) \equiv 0$ for all $u\in\N$, where $B(u)=\mathbb{P}\big((u\cdot\nabla)u\big)$.
  \end{enumerate}
\end{lemma}

\begin{proof}
  \textbf{1.  Characterisation of the kernel.}\;
  Suppose $u\in\N$, i.e.\ $u\in H$ and $A u = 0$.
  By definition of the Leray projector, $\mathbb{P}(\partial_1^2 u) = 0$ means that $\partial_1^2 u$ is a gradient field,
  there exists a scalar function $q$ such that
  \begin{equation}\label{eq:grad_pr}
    \partial_1^2 u = \nabla q \qquad\text{a.e.\ on }\mathbb{T}^2.
  \end{equation}
  Taking the divergence of \eqref{eq:grad_pr} and using $\nabla\cdot u = 0$ gives
  \[
    \Delta q = \partial_1(\partial_1^2 u_1) + \partial_2(\partial_1^2 u_2)
    = \partial_1^2(\partial_1 u_1 + \partial_2 u_2) = 0 .
  \]
   By multiplying $\Delta q=0$ by $q$ and integrate by parts
  to obtain $\int_{\mathbb{T}^2} |\nabla q|^2 dx = 0$,  then the harmonic functions $q$ are constants.
 Eq. \eqref{eq:grad_pr} reduces to $\partial_1^2 u = 0$.

  Now, for a.e.\ fixed $x_2$, $u(\cdot,x_2)$ is a $C^1$ vector-valued function of $x_1$ satisfying
  $\partial_1^2 u = 0$; hence it is an affine function in $x_1$:
  \[
    u(x_1,x_2) = a(x_2) + x_1 b(x_2)
  \]
  for some vector fields $a,b$.  Periodicity in $x_1$ forces $u(0,x_2)=u(2\pi,x_2)$, which yields
  $b(x_2) \equiv 0$.  Therefore $u$ does not depend on $x_1$, i.e.\ $\partial_1 u = 0$.
  This proves the first claim.

  Conversely, if $\partial_1 u = 0$ in $H$, then $\partial_1^2 u = 0$ and $A u = 0$, so $u\in\N$.
  Hence $\N = \{ u\in H : \partial_1 u = 0\}$.

  \textbf{2.  Explicit description of $\N$.}\;
  Let $u\in\N$.  Because $\partial_1 u = 0$, we can write
  \[
    u_1(x_1,x_2) = a(x_2),\qquad u_2(x_1,x_2) = b(x_2)
  \]
  for some $a,b\in L^2(\mathbb{T})$.  The incompressibility condition
  $\partial_1 u_1 + \partial_2 u_2 = 0$ gives $b'(x_2) = 0$, so $b$ is constant.
  Since $u$ belongs to $H$, it has zero mean, which forces
  \[
    0 = \int_{\mathbb{T}^2} u_2 \,dx = (2\pi)^2 b \;\Longrightarrow\; b = 0 .
  \]
  Hence $u_2 \equiv 0$.  The zero-mean condition on $u_1$ becomes $\int_{\mathbb{T}} a(x_2)\,dx_2 = 0$.
  We obtain the form \eqref{eq:kernel_form}.
  The space is infinite-dimensional because, for instance, the functions
  \[
    e_k(x_1,x_2) = \frac{1}{\pi}\begin{pmatrix} \sin(k x_2) \\ 0 \end{pmatrix},\quad k=1,2,\dots,
  \]
  belong to $\N$ and are orthonormal in $L^2$.

  \textbf{3.  Vanishing of the nonlinear term.}\;
  Take $u\in\N$ given by \eqref{eq:kernel_form}.  Then
  \[
    u\cdot\nabla = f(x_2)\partial_1 + 0\cdot\partial_2 = f(x_2)\partial_1 .
  \]
  Applying this operator to $u$ yields
  \[
    (u\cdot\nabla)u = f(x_2)\partial_1 \begin{pmatrix} f(x_2) \\ 0 \end{pmatrix}
    = \begin{pmatrix} f(x_2)\,\partial_1 f(x_2) \\ 0 \end{pmatrix}
    = \begin{pmatrix} 0 \\ 0 \end{pmatrix},
  \]
  because $f(x_2)$ does not depend on $x_1$.  Hence $(u\cdot\nabla)u = 0$, and
  $B(u) = \mathbb{P}(0) = 0$.
\end{proof}

\begin{remark}[Comparison of kernels in different settings]\label{rmk:kernel_comparison}
  It is instructive to compare the kernel of the Stokes operator under three different
  boundary conditions or dissipation regimes.
  \begin{enumerate}
   \item \textbf{Full dissipation on $\mathbb{T}^2$.}
  If $A_{\mathrm{full}} = -\PP\Delta$, then $A_{\mathrm{full}} u = 0$ implies $\Delta u = \nabla q$.
  Taking the divergence gives $\Delta q = 0$, so $q$ is constant on the torus and $\Delta u = 0$.
  Multiplying $\Delta u = 0$ by $u$, integrating by parts over $\mathbb{T}^2$ and using the
  periodic boundary conditions yields
  \[
    0 = \int_{\mathbb{T}^2} u \cdot \Delta u \,dx = -\int_{\mathbb{T}^2} |\nabla u|^2 \,dx .
  \]
  Hence $\nabla u \equiv 0$, which forces $u$ to be a constant vector.  The zero-mean condition
  $\int_{\mathbb{T}^2} u\,dx = 0$ then gives $u \equiv 0$.  Thus $\ker A_{\mathrm{full}} = \{0\}$.
    \item \textbf{Horizontal dissipation with Dirichlet (no-flux) boundary condition.}
      On a bounded domain $\Omega=[0,L_x]\times[0,L_y]$ with $u=0$ on $\partial\Omega$,
      the equation $-\mathbb{P}\partial_1^2 u = 0$ together with the boundary condition forces
      $u\equiv 0$.  Indeed, $\partial_1^2 u = 0$ implies $u$ is affine in $x_1$:
      $u(x_1,x_2) = a(x_2) + x_1\,b(x_2)$.  The boundary condition $u(0,x_2)=u(L_x,x_2)=0$
      gives $a\equiv b\equiv 0$.  Hence $\ker A_{\mathrm{Dir}} = \{0\}$ as well.
    \item \textbf{Horizontal dissipation on $\mathbb{T}^2$ (the present setting).}
      In the periodic case, $\partial_1^2 u = 0$ only forces $u$ to be independent of $x_1$,
      but imposes no restriction on the $x_2$-dependence.  Divergence-free and zero-mean
      conditions then lead to the infinite-dimensional kernel described in Lemma~\ref{lem:kernel}.
  \end{enumerate}
  The essential difference is that the full Laplacian provides dissipation in \emph{all} spatial
  directions, while the horizontal Laplacian leaves purely vertical variations untouched.
  Whether these variations survive as admissible steady states depends on the boundary condition:
  Dirichlet side-walls kill them via the no-slip requirement, whereas periodic boundary conditions
  allow them to persist.  This observation explains why the anisotropic Navier-Stokes equations on
  the torus possess infinitely many steady deterministic states and why exponential mixing cannot
  hold (see Section~\ref{sec:mixing}).
\end{remark}

\subsection{Stochastic setting}\label{sec:stochastic}

Let $(\Omega,\mathcal{F},\{\mathcal{F}_t\}_{t\geq0},\mathbb{P})$ be a complete filtered probability
space.
Let $H$ be the Hilbert space defined in \eqref{eq:Hdef}. A \emph{cylindrical Wiener process} $W$ on $H$ is a family of continuous linear maps
$W(t):H\to L^2(\Omega)$ such that for every $h\in H$, $\{W(t)h\}_{t\ge0}$ is a real-valued
$\mathcal{F}_t$-Brownian motion with variance $t\|h\|_H^2$.
There exists a separable Hilbert space $U$ such that the embedding $H\hookrightarrow U$
is Hilbert-Schmidt, and $W$ can be identified with a $U$-valued Wiener process whose
covariance operator is determined by this embedding (see \cite{DPZ,LR}).
In particular, the trajectories of $W$ belong to $C([0,T];U)$ almost surely.

The stochastic perturbation in \eqref{eq:SNSintro} is given by $Q^{1/2}dW$, where
$Q:H\to H$ is a non-negative, symmetric, trace-class operator.
We assume the following regularity:
\begin{equation}\label{eq:Qreg}
  Q^{1/2}: H\longrightarrow \widetilde H^{1,1}\quad\text{is Hilbert-Schmidt}.
\end{equation}
Condition \eqref{eq:Qreg} guarantees that the linear Ornstein-Uhlenbeck process
(see Section~\ref{sec:linear}) possesses enough regularity to make the compactness
arguments work.

\section{Linear LDP in \(G\)}\label{sec:linear}

Define the auxiliary space
\begin{align*}
  &G := C([0,T];L^2)\cap L^2(0,T;\widetilde H^{1,1}),\\
  &\|z\|_{G} = \sup_{0\le t\le T}\|z(t)\|_{L^2} + \biggl(\int_0^T \|z(t)\|_{\widetilde H^{1,1}}^2\,dt\biggr)^{1/2}.
\end{align*}

Consider the linear stochastic Stokes equation
\begin{equation}\label{eq:lin}
  dz^\varepsilon + A z^\varepsilon dt = \sqrt\varepsilon\, Q^{1/2} dW,\qquad z^\varepsilon(0)=0 .
\end{equation}
Because the solution map is measurable, there exists \(\mathcal G^\varepsilon: C([0,T];U)\to G\) such that \(z^\varepsilon = \mathcal G^\varepsilon(W)\) a.s.

For \(\varphi\in L^2(0,T;H)\), the skeleton equation is
\begin{equation}\label{eq:ske}
  \partial_t z^\varphi + A z^\varphi = Q^{1/2}\varphi,\qquad z^\varphi(0)=0 .
\end{equation}
Define \(\mathcal G^0(\int_0^\cdot \varphi) = z^\varphi\).

\begin{lemma}[A priori estimates for the skeleton equation]\label{lem:energy}
  Let $z^\varphi$ be the solution of \eqref{eq:ske} with $\varphi\in L^2(0,T;H)$.
  Under the assumption that $Q^{1/2}:H\to\widetilde H^{1,1}$ is a bounded linear operator,
  there exists a constant $C=C(T,Q)>0$ such that
  \[
    \|z^\varphi\|_{G}^2
    = \sup_{0\le t\le T}\|z^\varphi(t)\|_{L^2}^2
      + \int_0^T \|z^\varphi(t)\|_{\widetilde H^{1,1}}^2\,dt
    \le C \int_0^T \|\varphi(t)\|_{H}^2\,dt .
  \]
\end{lemma}

\begin{proof}
  The estimates can be justified by a standard Galerkin approximation. We split the proof into two steps.

  \medskip
  \noindent\textit{Step 1.~$L^2$ energy estimate.}
  Taking the $L^2$ inner product of \eqref{eq:ske} with $z^\varphi$ and using
  $\langle\mathbb{P}\partial_1^2 z^\varphi, z^\varphi\rangle = -\|\partial_1 z^\varphi\|_{L^2}^2$,
  we obtain
  \begin{equation}\label{eq:linL2}
    \frac12\frac{d}{dt}\|z^\varphi\|_{L^2}^2 + \|\partial_1 z^\varphi\|_{L^2}^2
    = \langle Q^{1/2}\varphi, z^\varphi\rangle .
  \end{equation}
  The right-hand side is bounded by
  \[
    |\langle Q^{1/2}\varphi, z^\varphi\rangle|
    \le \|Q^{1/2}\varphi\|_{L^2} \|z^\varphi\|_{L^2}
    \le \frac12 \|z^\varphi\|_{L^2}^2 + \frac12 \|Q^{1/2}\varphi\|_{L^2}^2 .
  \]
  Because $Q^{1/2}:H\to H$ is bounded (it is even Hilbert-Schmidt), we have
  $\|Q^{1/2}\varphi\|_{L^2} \le C_Q \|\varphi\|_{H}$.  Hence
  \[
    \frac{d}{dt}\|z^\varphi\|_{L^2}^2 + 2\|\partial_1 z^\varphi\|_{L^2}^2
    \le \|z^\varphi\|_{L^2}^2 + C \|\varphi\|_{H}^2 .
  \]
  Dropping the nonnegative term $2\|\partial_1 z^\varphi\|_{L^2}^2$ and applying Gronwall's
  inequality gives, for all $t\in[0,T]$,
  \begin{equation}\label{eq:L2bound}
    \|z^\varphi(t)\|_{L^2}^2
    \le \int_0^t e^{t-s} C \|\varphi(s)\|_{H}^2\,ds
    \le C(T) \int_0^T \|\varphi(s)\|_{H}^2\,ds .
  \end{equation}
  Integrating \eqref{eq:linL2} in time and using \eqref{eq:L2bound} yields
  \begin{equation}\label{eq:L2diss}
    \int_0^T \|\partial_1 z^\varphi(t)\|_{L^2}^2\,dt
    \le C(T) \int_0^T \|\varphi(t)\|_{H}^2\,dt .
  \end{equation}

  \medskip
  \noindent\textit{Step 2.~Estimate of the vertical derivative.}
  Apply $\partial_2$ to \eqref{eq:ske}, we obtain
  \[
    \partial_t \partial_2 z^\varphi + A \partial_2 z^\varphi = \partial_2 Q^{1/2}\varphi .
  \]
  Taking the $L^2$ inner product with $\partial_2 z^\varphi$ gives
  \begin{equation}\label{eq:linH1}
    \frac12\frac{d}{dt}\|\partial_2 z^\varphi\|_{L^2}^2 + \|\partial_1\partial_2 z^\varphi\|_{L^2}^2
    = \langle \partial_2 Q^{1/2}\varphi, \partial_2 z^\varphi\rangle .
  \end{equation}
  The right-hand side is estimated by
  \[
    |\langle \partial_2 Q^{1/2}\varphi, \partial_2 z^\varphi\rangle|
    \le \|\partial_2 Q^{1/2}\varphi\|_{L^2} \|\partial_2 z^\varphi\|_{L^2}
    \le \frac12 \|\partial_2 z^\varphi\|_{L^2}^2 + \frac12 \|\partial_2 Q^{1/2}\varphi\|_{L^2}^2 .
  \]
  By hypothesis $Q^{1/2}:H\to\widetilde H^{1,1}$ is bounded, so $\|\partial_2 Q^{1/2}\varphi\|_{L^2}
  \le \|Q^{1/2}\varphi\|_{\widetilde H^{1,1}} \le C \|\varphi\|_{H}$.  Thus
  \[
    \frac{d}{dt}\|\partial_2 z^\varphi\|_{L^2}^2 + 2\|\partial_1\partial_2 z^\varphi\|_{L^2}^2
    \le \|\partial_2 z^\varphi\|_{L^2}^2 + C \|\varphi\|_{H}^2 .
  \]
  As before, Gronwall's inequality yields
  \[
    \sup_{0\le t\le T}\|\partial_2 z^\varphi(t)\|_{L^2}^2
    \le C(T) \int_0^T \|\varphi(t)\|_{H}^2\,dt ,
  \]
  and integrating \eqref{eq:linH1} over $[0,T]$ gives
  \begin{equation}\label{eq:vertical}
    \int_0^T \|\partial_1\partial_2 z^\varphi(t)\|_{L^2}^2\,dt
    \le C(T) \int_0^T \|\varphi(t)\|_{H}^2\,dt .
  \end{equation}

  The $\widetilde H^{1,1}$ norm satisfies
  \[
    \|z^\varphi\|_{\widetilde H^{1,1}}^2
    = \|z^\varphi\|_{L^2}^2 + \|\partial_1 z^\varphi\|_{L^2}^2
      + \|\partial_2 z^\varphi\|_{L^2}^2 + \|\partial_1\partial_2 z^\varphi\|_{L^2}^2 .
  \]
  Combining \eqref{eq:L2bound}, \eqref{eq:L2diss},
  the estimate for $\|\partial_2 z^\varphi\|_{L^2}$ and \eqref{eq:vertical}, we obtain
  \[
    \sup_{0\le t\le T}\|z^\varphi(t)\|_{L^2}^2 + \int_0^T \|z^\varphi(t)\|_{\widetilde H^{1,1}}^2\,dt
    \le C(T,Q) \int_0^T \|\varphi(t)\|_{H}^2\,dt .
  \]
  This is precisely the desired bound in the $G$-norm.
\end{proof}

\begin{lemma}[Compactness of the skeleton set]\label{lem:compact}
  For every $N<\infty$, the set
  \[
    K_N = \Bigl\{\mathcal G^0\Bigl(\int_0^\cdot \varphi(s)\,ds\Bigr) :
      \varphi\in L^2(0,T;H),\;\int_0^T \|\varphi(t)\|_H^2\,dt \le N\Bigr\}
  \]
  is a compact subset of $G = C([0,T];L^2)\cap L^2(0,T;\widetilde H^{1,1})$.
\end{lemma}

\begin{proof}
  To prove compactness of $K_N$, it suffices to show that for any sequence
  $\{\varphi_n\}\subset S_N:=\{\varphi:\int_0^T\|\varphi\|_H^2\le N\}$, there exists a subsequence
  (still denoted by $\varphi_n$) and a limit $\varphi\in S_N$ such that
  \[
    z^{\varphi_n}\longrightarrow z^\varphi \quad\text{strongly in } G.
  \]
  Since $S_N$ is weakly compact in the Hilbert space $L^2(0,T;H)$,
  we may extract a subsequence with
  \[
    \varphi_n \rightharpoonup \varphi \quad\text{weakly in } L^2(0,T;H).
  \]

  \medskip
  Define the operator $\Phi: L^2(0,T;H) \to C([0,T];\widetilde H^{1,1})$ by
  \[
    \Phi(v)(t) = \int_0^t Q^{1/2}v(s)\,ds,\qquad v\in L^2(0,T;H).
  \]
  Because $Q^{1/2}:H\to\widetilde H^{1,1}$ is a Hilbert-Schmidt operator, $\Phi$ is a compact
  linear operator from $L^2(0,T;H)$ (endowed with the weak topology) into
  $C([0,T];\widetilde H^{1,1})$ (endowed with the strong topology).
  Indeed, for every fixed $t$, the map $v\mapsto \int_0^t Q^{1/2}v(s)ds$ is Hilbert-Schmidt,
  and the family $\{\Phi(v_n)\}$ is equi-continuous in $C([0,T];\widetilde H^{1,1})$ by
  the a priori bound $\|\Phi(v)(t)-\Phi(v)(s)\|_{\widetilde H^{1,1}} \le |t-s|^{1/2} \|Q^{1/2}\|_{L_2(H,\widetilde H^{1,1})} \|v\|_{L^2(H)}$.
  The Arzel\`{a}-Ascoli theorem then yields the compactness (see e.g.\ \cite{Wang2023}).
  Consequently, the weak convergence $\varphi_n\rightharpoonup\varphi$ implies
  \begin{equation}\label{eq:Phi_conv}
    \Phi(\varphi_n) \longrightarrow \Phi(\varphi) \quad\text{strongly in } C([0,T];\widetilde H^{1,1}).
  \end{equation}

  Set $w_n = z^{\varphi_n} - z^\varphi$. Subtracting the equations for $\varphi_n$ and $\varphi$,
  we obtain
  \begin{equation}\label{eq:diff_ske}
    \partial_t w_n + A w_n = Q^{1/2}(\varphi_n-\varphi),\qquad w_n(0)=0.
  \end{equation}
  Taking the $L^2$ inner product with $w_n$ gives
  \begin{equation}\label{eq:energy_diff}
    \frac12\frac{d}{dt}\|w_n\|_{L^2}^2 + \|\partial_1 w_n\|_{L^2}^2
    = \langle Q^{1/2}(\varphi_n-\varphi), w_n\rangle .
  \end{equation}
  To handle the right-hand side, we use the primitive $\Phi$ and integrate by parts:
  \[
    \langle Q^{1/2}(\varphi_n-\varphi), w_n\rangle
    = \frac{d}{dt}\langle \Phi(\varphi_n-\varphi), w_n\rangle
      - \langle \Phi(\varphi_n-\varphi), \partial_t w_n\rangle .
  \]
  Insert $\partial_t w_n = -A w_n + Q^{1/2}(\varphi_n-\varphi)$ into the second term:
  \begin{align*}
    \langle Q^{1/2}(\varphi_n-\varphi), w_n\rangle
    &= \frac{d}{dt}\langle \Phi_n, w_n\rangle
       - \langle \Phi_n, -A w_n + Q^{1/2}(\varphi_n-\varphi)\rangle \\
    &= \frac{d}{dt}\langle \Phi_n, w_n\rangle
       + \langle \Phi_n, A w_n\rangle
       - \langle \Phi_n, Q^{1/2}(\varphi_n-\varphi)\rangle ,
  \end{align*}
  where we write $\Phi_n = \Phi(\varphi_n-\varphi)$ for brevity.
  Since $A = -\mathbb{P}\partial_1^2$ is symmetric,
  \[
    \langle \Phi_n, A w_n\rangle
    = \langle A \Phi_n, w_n\rangle
    = -\langle \partial_1^2 \Phi_n, w_n\rangle
    = \langle \partial_1 \Phi_n, \partial_1 w_n\rangle .
  \]

  Substituting back into \eqref{eq:energy_diff} and integrating in time, we get
  \begin{align}\label{eq:integral_identity}
    \frac12\|w_n(t)\|_{L^2}^2 &+ \int_0^t \|\partial_1 w_n\|_{L^2}^2\,ds
      = \langle \Phi_n(t), w_n(t)\rangle \notag\\
    &\quad + \int_0^t \langle \partial_1 \Phi_n, \partial_1 w_n\rangle\,ds
      - \int_0^t \langle \Phi_n, Q^{1/2}(\varphi_n-\varphi)\rangle\,ds .
  \end{align}
  (The boundary term at $s=0$ vanishes because $w_n(0)=0$ and $\Phi_n(0)=0$.)

  Now we estimate each term on the right-hand side.  Recall from \eqref{eq:Phi_conv} that
  \begin{equation}\label{eq:Phi_small}
    \sup_{0\le t\le T}\|\Phi_n(t)\|_{\widetilde H^{1,1}} \longrightarrow 0 \quad\text{as } n\to\infty .
  \end{equation}
    Since $\|u\|_{L^2} \le \|u\|_{\widetilde H^{1,1}}$ and
  $\|\partial_1 u\|_{L^2} \le \|u\|_{\widetilde H^{1,1}}$ for every $u\in\widetilde H^{1,1}$,
  the convergence \eqref{eq:Phi_conv} implies
  \[
    \|\Phi_n\|_{C([0,T];L^2)}
    = \sup_{0\le t\le T}\|\Phi_n(t)\|_{L^2}
    \le \sup_{0\le t\le T}\|\Phi_n(t)\|_{\widetilde H^{1,1}} \longrightarrow 0,
  \]
  and similarly $\|\partial_1 \Phi_n\|_{C([0,T];L^2)} \to 0$.

 For the first term,
      \[
        |\langle \Phi_n(t), w_n(t)\rangle|
        \le \|\Phi_n(t)\|_{L^2} \|w_n(t)\|_{L^2}
        \le \frac14 \|w_n(t)\|_{L^2}^2 + \|\Phi_n(t)\|_{L^2}^2 .
      \]
 For the second term,
      \begin{align*}
        \int_0^t |\langle \partial_1 \Phi_n, \partial_1 w_n\rangle|\,ds
        &\le \int_0^t \|\partial_1 \Phi_n\|_{L^2} \|\partial_1 w_n\|_{L^2}\,ds\\
        &\le \frac12 \int_0^t \|\partial_1 w_n\|_{L^2}^2\,ds
          + \frac12 \int_0^t \|\partial_1 \Phi_n\|_{L^2}^2\,ds .
      \end{align*}
 For the last term,
      \begin{align*}
        \int_0^t |\langle \Phi_n, Q^{1/2}(\varphi_n-\varphi)\rangle|\,ds
        &\le \int_0^t \|\Phi_n\|_{L^2} \|Q^{1/2}(\varphi_n-\varphi)\|_{L^2}\,ds\\
        &\le \Bigl(\sup_{s\le T}\|\Phi_n(s)\|_{L^2}\Bigr)
          \int_0^T \|Q^{1/2}(\varphi_n-\varphi)\|_{L^2}\,ds .
     \end{align*}
      Since $Q^{1/2}:H\to L^2$ is bounded and $\{\varphi_n-\varphi\}$ is weakly convergent hence bounded in $L^2(H)$,
      the integral $\int_0^T \|Q^{1/2}(\varphi_n-\varphi)\|_{L^2}\,ds$ is uniformly bounded, together with
      $\sup\|\Phi_n\|_{L^2}\to0$ we get that this term tends to zero.

  Absorbing the $\|\partial_1 w_n\|_{L^2}^2$ and $\|w_n\|_{L^2}^2$ terms into the left-hand side,
   we obtain for all $t\in[0,T]$
  \begin{equation}\label{eq:gronwall_temp}
    \|w_n(t)\|_{L^2}^2 + \int_0^t \|\partial_1 w_n\|_{L^2}^2\,ds
    \le C_1 \int_0^t \|w_n(s)\|_{L^2}^2\,ds + \varepsilon_n,
  \end{equation}
  where $\varepsilon_n\to0$ as $n\to\infty$ (collecting all the terms that involve $\Phi_n$ and vanish
  due to \eqref{eq:Phi_small}).  Applying Gronwall's inequality to \eqref{eq:gronwall_temp} yields
  \[
    \sup_{0\le t\le T}\|w_n(t)\|_{L^2}^2 \le \varepsilon_n \, e^{C_1 T} \longrightarrow 0 .
  \]
  Hence $w_n\to0$ strongly in $C([0,T];L^2)$.  Plugging this back into \eqref{eq:gronwall_temp} also gives
  $\int_0^T \|\partial_1 w_n\|_{L^2}^2\,dt\to0$.

  It remains to verify strong convergence in the $\widetilde H^{1,1}$-norm.  Apply $\partial_2$ to the
  difference equation \eqref{eq:diff_ske} and take the $L^2$ inner product with $\partial_2 w_n$.  An
  analogous computation using the fact that $\Phi(\varphi_n-\varphi)\to0$ in $C([0,T];\widetilde H^{1,1})$
  (which also implies $\|\partial_2 \Phi_n\|_{C(L^2)}\to0$ and $\|\partial_1\partial_2\Phi_n\|_{C(L^2)}\to0$)
  shows that
  \[
    \sup_{t\le T}\|\partial_2 w_n(t)\|_{L^2}^2 + \int_0^T \|\partial_1\partial_2 w_n\|_{L^2}^2\,dt \longrightarrow 0 .
  \]
  Consequently,
  \[
    \|w_n\|_{G}^2 = \sup_{t}\|w_n\|_{L^2}^2 + \int_0^T \|w_n\|_{\widetilde H^{1,1}}^2\,dt \longrightarrow 0 .
  \]
  This proves that $\{z^{\varphi_n}\}$ converges strongly in $G$ to $z^{\varphi}$.

  We have shown that every sequence in $S_N$ admits a subsequence for which the corresponding solutions
  converge in $G$ to a limit of the same form.  Hence $K_N$ is sequentially compact, and since $G$ is a metric space,
  sequential compactness is equivalent to compactness.  The proof is complete.
\end{proof}

Consider controlled processes. For \(v^\varepsilon\in\mathcal A_N\) (predictable processes with \(\int_0^T\|v^\varepsilon\|^2\le N\) a.s.), define
\[
  z^\varepsilon_{v^\varepsilon} = \mathcal G^\varepsilon\Bigl(W + \frac1{\sqrt\varepsilon}\int_0^\cdot v^\varepsilon(s)ds\Bigr),
\]
which satisfies
\begin{align}\label{eq:lin_controlled}
  dz^\varepsilon_{v^\varepsilon} + A z^\varepsilon_{v^\varepsilon} dt = Q^{1/2}v^\varepsilon dt + \sqrt\varepsilon\,Q^{1/2}dW,\quad z^\varepsilon_{v^\varepsilon}(0)=0 .
\end{align}

\begin{lemma}[Convergence of the controlled linear processes]\label{lem:conv}
  Let $\{v^\varepsilon\}_{\varepsilon>0}\subset\mathcal A_N$ converge in distribution
  (as $S_N$-valued random variables) to $v\in\mathcal A_N$.  Then
  \[
    z^\varepsilon_{v^\varepsilon}:= \mathcal G^\varepsilon\Bigl(W + \frac{1}{\sqrt\varepsilon}\int_0^\cdot v^\varepsilon(s)ds\Bigr)
    \xrightarrow{\;d\;} z^v \quad\text{in } G,
  \]
  where $G = C([0,T];L^2)\cap L^2(0,T;\widetilde H^{1,1})$.
\end{lemma}

\begin{proof}
  We split the proof into three main parts.

  \medskip
  \noindent\textbf{Part~1. Uniform moment bounds.}
  By the Girsanov theorem, the process $z^\varepsilon_{v^\varepsilon}$ is the unique strong solution of the
  linear equation
  \begin{equation}\label{eq:lin_ctrl}
    dz^\varepsilon_{v^\varepsilon} + A z^\varepsilon_{v^\varepsilon}\,dt
    = Q^{1/2} v^\varepsilon dt + \sqrt\varepsilon\, Q^{1/2} dW, \qquad
    z^\varepsilon_{v^\varepsilon}(0)=0 .
  \end{equation}
  We shall derive uniform estimates for $\mathbb E\|z^\varepsilon_{v^\varepsilon}\|_G^2$.

  \textit{Estimate of $\|z^\varepsilon_{v^\varepsilon}\|_{L^2}^2$.}
  Applying It\^o's formula to $\|z^\varepsilon_{v^\varepsilon}(t)\|_{L^2}^2$ gives
  \begin{align}\label{ito1}
    d\|z\|^2 + 2\|\partial_1 z\|^2 dt
    &= 2\langle z, Q^{1/2} v^\varepsilon\rangle dt
      + 2\sqrt\varepsilon \langle z, Q^{1/2} dW\rangle
      + \varepsilon\,\mathrm{Tr}(Q)\,dt .
  \end{align}
  Take expectation.  By the Cauchy-Schwarz and Young inequalities,
  \[
    \mathbb E|\langle z, Q^{1/2} v^\varepsilon\rangle|
    \le \tfrac12 \mathbb E\|z\|^2 + \tfrac12 \mathbb E\|Q^{1/2} v^\varepsilon\|^2 .
  \]
  Since $Q^{1/2}:H\to H$ is bounded, we have
  $\mathbb E\|Q^{1/2} v^\varepsilon\|^2 \le C_Q\, \mathbb E\|v^\varepsilon\|^2$.
  Moreover, $\mathbb E\int_0^T\|v^\varepsilon\|^2 dt\le N$ by assumption.
  Discarding the non-negative term $2\mathbb E\|\partial_1 z\|^2$ and using
  $0<\varepsilon\le1$, we obtain the differential inequality
  \begin{equation}\label{ineq1}
    \frac{d}{dt}\mathbb E\|z\|^2 \le \mathbb E\|z\|^2 + C(N,Q) .
  \end{equation}
  Gronwall's lemma yields, for all $t\in[0,T]$,
  \begin{equation}\label{L2bound}
    \mathbb E\|z^\varepsilon_{v^\varepsilon}(t)\|_{L^2}^2 \le C(T,N,Q) .
  \end{equation}
  Furthermore, applying the Burkholder-Davis-Gundy inequality to the stochastic
  integral in \eqref{ito1} before taking expectation gives the pathwise bound
  \begin{equation}\label{supL2}
    \mathbb E\Bigl[ \sup_{0\le t\le T}\|z^\varepsilon_{v^\varepsilon}(t)\|_{L^2}^2 \Bigr]
    \le C(T,N,Q) .
  \end{equation}
  \textit{Estimate of $\|\partial_2 z^\varepsilon_{v^\varepsilon}\|_{L^2}^2$.}
  Apply the vertical derivative $\partial_2$ to \eqref{eq:lin_ctrl},
 hence $\partial_2 z$ satisfies
  \[
    d(\partial_2 z) + A(\partial_2 z)\,dt
    = \partial_2 Q^{1/2} v^\varepsilon dt + \sqrt\varepsilon\, \partial_2 Q^{1/2} dW .
  \]
  This is exactly the same form as \eqref{eq:lin_ctrl} with $z$ replaced by $\partial_2 z$
  and $Q^{1/2}$ replaced by $\partial_2 Q^{1/2}$.  The Hilbert-Schmidt assumption
  $Q^{1/2}:H\to\widetilde H^{1,1}$ guarantees that $\partial_2 Q^{1/2}: H\to L^2$ is
  also a bounded operator.  Thus, repeating the same argument
  as above we obtain
  \begin{equation}\label{supH1part}
    \mathbb E\Bigl[ \sup_{0\le t\le T}\|\partial_2 z^\varepsilon_{v^\varepsilon}(t)\|_{L^2}^2 \Bigr]
    \le C(T,N,Q) .
  \end{equation}
  In addition, integrating \eqref{ito1} over $[0,T]$ and using \eqref{L2bound} gives
  \begin{equation}\label{diss1}
    \mathbb E\int_0^T \|\partial_1 z^\varepsilon_{v^\varepsilon}(t)\|_{L^2}^2\,dt
    \le C(T,N,Q) .
  \end{equation}
  Similarly, the analog of \eqref{diss1} for $\partial_2 z$ provides
  \begin{equation}\label{diss2}
    \mathbb E\int_0^T \|\partial_1\partial_2 z^\varepsilon_{v^\varepsilon}(t)\|_{L^2}^2\,dt
    \le C(T,N,Q) .
  \end{equation}

  Collecting \eqref{supL2}, \eqref{supH1part}, \eqref{diss1} and \eqref{diss2}, and using
  the definition of the $\widetilde H^{1,1}$ norm, we conclude the uniform energy estimate for $\forall\,v^\varepsilon\in\mathcal A_N ,  \forall \varepsilon\in(0,1],$
  \begin{equation}\label{energyG}
    \mathbb E\Bigl[ \|z^\varepsilon_{v^\varepsilon}\|_G^2 \Bigr]
    = \mathbb E\Bigl[ \sup_{t}\|z^\varepsilon_{v^\varepsilon}\|_{L^2}^2
               + \int_0^T \|z^\varepsilon_{v^\varepsilon}\|_{\widetilde H^{1,1}}^2\,dt \Bigr]
    \le C(T,N,Q) .
  \end{equation}

  \medskip
  \textbf{Step 2. Decomposition of the controlled process.}
For each $\varepsilon>0$ and a given control $v^\varepsilon\in\mathcal A_N$, let
$z_{v^\varepsilon}$ be the unique solution of the \emph{deterministic} skeleton equation
\begin{equation}\label{eq:skeleton_v}
  \partial_t z_{v^\varepsilon} + A z_{v^\varepsilon} = Q^{1/2} v^\varepsilon,\qquad
  z_{v^\varepsilon}(0)=0 .
\end{equation}
Define the difference
\[
  m^\varepsilon(t) := z^\varepsilon_{v^\varepsilon}(t) - z_{v^\varepsilon}(t),\qquad t\in[0,T].
\]
Subtracting \eqref{eq:skeleton_v} from \eqref{eq:lin_ctrl} we obtain
\[
  dm^\varepsilon + A m^\varepsilon\,dt = \sqrt{\varepsilon}\,Q^{1/2}dW,\qquad m^\varepsilon(0)=0 .
\]
Applying It\^o's formula to $\|m^\varepsilon(t)\|_{L^2}^2$ gives
\[
  \frac12 \mathbb{E}\|m^\varepsilon(t)\|_{L^2}^2 + \mathbb{E}\int_0^t \|\partial_1 m^\varepsilon(s)\|_{L^2}^2\,ds
  = \frac{\varepsilon}{2}\, \mathrm{Tr}(Q)\,t .
\]
By the Burkholder-Davis-Gundy inequality, we deduce, for every $\varepsilon\in(0,1)$,
\begin{equation}\label{bd:diff_L2}
  \mathbb{E}\Bigl[ \sup_{0\le t\le T}\|m^\varepsilon(t)\|_{L^2}^2 \Bigr] \le C\,\varepsilon\,\mathrm{Tr}(Q)\,T .
\end{equation}
A completely analogous computation with $\partial_2 m^\varepsilon$ yields
\[
  \mathbb{E}\Bigl[ \sup_{0\le t\le T}\|\partial_2 m^\varepsilon(t)\|_{L^2}^2
    + \int_0^T \|\partial_1\partial_2 m^\varepsilon(t)\|_{L^2}^2\,dt \Bigr]
  \le C\,\varepsilon\,\mathrm{Tr}(Q)\,T .
\]
Combining this with \eqref{bd:diff_L2} and the definition of the $G$-norm we conclude
\begin{equation}\label{conv_diff_zero}
  \mathbb{E}\bigl[ \|m^\varepsilon\|_G^2 \bigr] \xrightarrow{\;\varepsilon\to0\;} 0,
\end{equation}
and therefore $m^\varepsilon \to 0$ in probability in $G$.

\medskip
\noindent\textbf{Step 3. Convergence of the skeleton part.}
The compactness result of Lemma~\ref{lem:compact} tells us that the map
\[
  T: S_N \ni \varphi \longmapsto z^\varphi \in G
\]
is continuous when $S_N$ is endowed with the weak topology of $L^2(0,T;H)$ and $G$
is given its strong topology. Indeed, if $\varphi_n\rightharpoonup\varphi$ weakly in
$L^2(0,T;H)$, then $z^{\varphi_n}\to z^\varphi$ strongly in $G$. Consequently, $T$ is
also continuous as a mapping from the metric space $(S_N, \text{weak})$ to $G$, hence
it preserves convergence in distribution.
By our assumption, $v^\varepsilon \xrightarrow{\;d\;} v$ as $S_N$-valued random
variables (where $S_N$ is equipped with the weak topology). Therefore
\begin{equation}\label{conv_skeleton}
  z_{v^\varepsilon} = T(v^\varepsilon) \;\xrightarrow{\;d\;}\; T(v) = z^v \quad\text{in } G .
\end{equation}

We have the decomposition
\[
  z^\varepsilon_{v^\varepsilon} = z_{v^\varepsilon} + m^\varepsilon,
\]
where $z_{v^\varepsilon} \xrightarrow{\;d\;} z^v$ in $G$ and $m^\varepsilon \xrightarrow{\;P\;} 0$ in $G$.
Since $G$ is a separable metric space, Slutsky's theorem (see e.g.\ \cite{Billingsley})
implies that $z^\varepsilon_{v^\varepsilon}$ converges in distribution to $z^v$ in $G$.
This completes the proof of Lemma~\ref{lem:conv}.
\end{proof}

\begin{theorem}[LDP for the linear equation]\label{th:linLDP}
  The family \(\{z^\varepsilon\}\) satisfies the LDP in \(G\) with good rate function
  \[
    I_0(z) = \frac12 \inf\Bigl\{ \int_0^T \|\varphi\|_{H}^2\,dt : \varphi\in L^2(0,T;H),\; z = z^\varphi \Bigr\}.
  \]
\end{theorem}
\begin{proof}
The proof is a direct application of the Budhiraja-Dupuis theorem \cite{BD,BDM} combined with Lemmas \ref{lem:compact} and \ref{lem:conv}.
\end{proof}

\section{Uniform Large Deviations}
\label{sec:nonlinear}

In this section we prove the solutions of the
anisotropic Navier--Stokes equations \eqref{eq:SNSintro} satisfy a uniform large deviation
principle in $C([0,T];H)$ with respect to initial data belonging to $\widetilde H^{0,1}$.

\subsection{The solution map and its regularity}
\label{sec:map}

For $z\in G = C([0,T];H)\cap L^2(0,T;\widetilde H^{1,1})$ and initial data
$x\in\widetilde H^{0,1}$, consider the deterministic auxiliary equation
\begin{equation}\label{eq:aux_v}
  \partial_t v + A v + \mathbb P\bigl(((v+z)\cdot\nabla)(v+z)\bigr) = 0,
  \qquad v(0)=x .
\end{equation}
By the anisotropic energy method of \cite{LZZ}, one obtains a unique solution
\[
v \in L^\infty(0,T;\widetilde H^{0,1}) \cap L^2(0,T;\widetilde H^{1,1}) .
\]
Because $v\in L^2(0,T;\widetilde H^{1,1})\subset L^2(0,T;H^1)$,
we have $\|A v\|_{H^{-1}} \le C\|v\|_{H^1}$, hence $A v \in L^2(0,T;H^{-1})$.
  Because $\nabla\cdot v = \nabla\cdot z = 0$, the velocity field
  $u := v+z$ is also divergence-free.  Hence the nonlinear term can be
  rewritten as a divergence:
  \[
    (u\cdot\nabla)u = \nabla\cdot(u\otimes u).
  \]
  For any test function $\phi\in H^1(\mathbb{T}^2;\mathbb{R}^2)$,
  integration by parts gives
  \[
    |\langle (u\cdot\nabla)u, \phi\rangle|
    = |\langle u\otimes u, \nabla\phi\rangle|
    \le \|u\otimes u\|_{L^2}\,\|\nabla\phi\|_{L^2}
    \le \|u\|_{L^4}^2\,\|\phi\|_{H^1}.
  \]
  Consequently,
  \begin{equation}\label{eq:H-1_est}
    \|(u\cdot\nabla)u\|_{H^{-1}} \le \|u\|_{L^4}^2\le C \|u\|_{L^2}\,\|\nabla u\|_{L^2}.
  \end{equation}
 From the a priori bounds
  $v\in L^\infty(0,T;\widetilde H^{0,1})\cap L^2(0,T;\widetilde H^{1,1})$
  and $z\in C([0,T];H)\cap L^2(0,T;\widetilde H^{1,1})$, we obtain
  \[
    u \in L^\infty(0,T;L^2) \cap L^2(0,T;H^1).
  \]
  Therefore, squaring \eqref{eq:H-1_est} and integrating in time,
  \begin{align*}
    \int_0^T \|(u\cdot\nabla)u\|_{H^{-1}}^2\,dt
    &\le C \int_0^T \|u\|_{L^2}^2\,\|\nabla u\|_{L^2}^2\,dt \\
    &\le C \|u\|_{L^\infty(0,T;L^2)}^2 \int_0^T \|\nabla u\|_{L^2}^2\,dt
    < \infty .
  \end{align*}
  Hence $(u\cdot\nabla)u \in L^2(0,T;H^{-1})$. Therefore $\partial_t v \in L^2(0,T;H^{-1})$.

Now we apply the classical Aubin-Lions compactness lemma (see \cite{Temam})
with the triple
\[
\widetilde H^{1,1} \hookrightarrow H \subset H^{-1}.
\]
(The compactness of $\widetilde H^{1,1}\hookrightarrow H$ follows from
$\widetilde H^{1,1}\subset H^1$ and the Rellich theorem, which implies
$H^1 \hookrightarrow L^2$ on the torus, restricting to the closed subspace
$\widetilde H^{1,1}$ preserves compactness.)
Since $v\in L^2(0,T;\widetilde H^{1,1})$ and $\partial_t v \in L^2(0,T;H^{-1})$,
the Aubin--Lions lemma yields $v\in C([0,T];H)$.

The nonlinear solution map $\mathcal F_x : G \to C([0,T];H)$ is then defined by
\[
\mathcal F_x(z) := z + v, \qquad z\in G .
\]
When $z=z^\varepsilon$ is the solution of the linear stochastic equation,
$u^\varepsilon = \mathcal F_x(z^\varepsilon)$ coincides with the solution of
the original stochastic Navier-Stokes equations.

\subsection{Uniform Lipschitz estimate}
\label{sec:lipschitz}

The following proposition is the key technical ingredient of the paper.
It shows that $\mathcal F_x$ is Lipschitz continuous on bounded subsets of $G$,
uniformly with respect to $x$ in bounded subsets of $\widetilde H^{0,1}$.

\begin{proposition}\label{prop:Lip}
  For every $R>0$, there exists a constant $L_R>0$ such that for all
  $z_1,z_2\in G$ with $\|z_i\|_G\le R$ and all $x$ with $\|x\|_{\widetilde H^{0,1}}\le R$,
  \[
    \|\mathcal F_x(z_1) - \mathcal F_x(z_2)\|_{C([0,T];H)} \le L_R \|z_1-z_2\|_G .
  \]
\end{proposition}

\begin{proof}
  Let $\zeta = z_1-z_2$, $u_i = \mathcal F_x(z_i)=z_i+v_i$ and
  $w = v_1-v_2$, $\delta = u_1-u_2 = \zeta + w$.
  Subtracting the equations \eqref{eq:aux_v} for $v_1,v_2$ we obtain
  \begin{equation}\label{eq:diff_w}
    \partial_t w + A w = - \mathbb P\bigl[(u_1\cdot\nabla)\delta
    + (\delta\cdot\nabla)u_2\bigr], \qquad w(0)=0 .
  \end{equation}
  Take the $L^2$ inner product of \eqref{eq:diff_w} with $w$ and use the
  cancellation $\langle (u_1\cdot\nabla)w, w\rangle =0$
  to obtain the energy identity
  \begin{align}\label{eq:energy_w}
    \frac12 \frac{d}{dt}\|w\|_{L^2}^2 + \|\partial_1 w\|_{L^2}^2
    &= -\langle (u_1\cdot\nabla)\zeta, w\rangle
      - \langle (w\cdot\nabla)u_2, w\rangle
      - \langle (\zeta\cdot\nabla)u_2, w\rangle \notag\\
    &=: I_1 + I_2 + I_3 .
  \end{align}
  We now estimate $I_1,I_2,I_3$ using Lemma~\ref{lem:GN} and the
  divergence-free condition.

  For $I_1$, using H\"older's inequality,
\begin{align*}
  |I_1| &\le \Bigl( \|u_{1,1}\|_{L^\infty_h L^2_v} \|\partial_1\zeta\|_{L^2_h L^\infty_v}
          + \|u_{1,2}\|_{L^2_h L^\infty_v} \|\partial_2\zeta\|_{L^\infty_h L^2_v} \Bigr) \|w\|_{L^2}.
\end{align*}
For the first factor, we apply the Lemma~\ref{lem:GN} to each component of $u_1$:
\begin{align*}
  \|u_{1,1}\|_{L^\infty_h L^2_v}
  &\le C\Bigl(\|u_{1,1}\|_{L^2}^{1/2}\,\|\partial_1 u_{1,1}\|_{L^2}^{1/2}
            + \|u_{1,1}\|_{L^2}\Bigr),\\
  \|u_{1,2}\|_{L^2_h L^\infty_v}
  &\le C\Bigl(\|u_{1,2}\|_{L^2}^{1/2}\,\|\partial_2 u_{1,2}\|_{L^2}^{1/2}
            + \|u_{1,2}\|_{L^2}\Bigr) \\
  &\le C\Bigl(\|u_1\|_{L^2}^{1/2}\,\|\partial_1 u_1\|_{L^2}^{1/2}
            + \|u_1\|_{L^2}\Bigr),
\end{align*}
where in the last line we used the relation $\partial_2 u_{1,2} = -\partial_1 u_{1,1}$
and the fact that $\|\partial_1 u_1\|_{L^2}$ controls
both horizontal derivatives.  The derivatives of $\zeta$ are handled similarly:
\begin{align*}
  \|\partial_1\zeta\|_{L^2_h L^\infty_v}
  &\le C\bigl(\|\partial_1\zeta\|_{L^2}^{1/2}\,\|\partial_1\partial_2\zeta\|_{L^2}^{1/2}
      + \|\partial_1\zeta\|_{L^2}\bigr),\\
\|\partial_2\zeta\|_{L^\infty_h L^2_v}
  &\le C\bigl(\|\partial_2\zeta\|_{L^2}^{1/2}\,\|\partial_1\partial_2\zeta\|_{L^2}^{1/2}
      + \|\partial_2\zeta\|_{L^2}\bigr).
\end{align*}

Inserting these bounds into $|I_1|$ and applying Young's inequality, we obtain,
\begin{align}\label{est_I1}
  |I_1| &\le C \Bigl(1+\| u_1\|_{L^2}^2+\|\partial_1 u_1\|_{L^2}^2
          +\|\partial_1\partial_2 u_1\|_{L^2}^2\Bigr) \|w\|_{L^2}^2 \notag\\
  &\quad + C\Bigl(\|\partial_2\zeta\|_{L^2}^2
          + \|\partial_1\zeta\|_{L^2}^2 + \|\partial_1\partial_2\zeta\|_{L^2}^2\Bigr).
\end{align}

  For $I_2$ we have
  \[
    |I_2| \le \bigl(\|w_1\|_{L^\infty_h L^2_v}\|\partial_1 u_2\|_{L^2_h L^\infty_v}
          + \|w_2\|_{L^2_h L^\infty_v}\|\partial_2 u_2\|_{L^\infty_h L^2_v}\bigr) \|w\|_{L^2}.
  \]
  Applying Lemma~\ref{lem:GN} and the identity $\partial_2 w_2=-\partial_1 w_1$,
  together with the Young inequality, we obtain for any $\eta>0$
  \begin{equation}\label{est_I2}
    |I_2| \le \eta \|\partial_1 w\|_{L^2}^2
      + C_\eta \Bigl(1+\|\partial_1 u_2\|_{L^2}^2+\|\partial_2 u_2\|_{L^2}^2
        +\|\partial_1\partial_2 u_2\|_{L^2}^2\Bigr) \|w\|_{L^2}^2 .
  \end{equation}
  For $I_3$, we have
  \begin{align*}
    |I_3| &\le \bigl(\|\zeta_1\|_{L^\infty_h L^2_v}\|\partial_1 u_2\|_{L^2_h L^\infty_v}
          + \|\zeta_2\|_{L^2_h L^\infty_v}\|\partial_2 u_2\|_{L^\infty_h L^2_v}\bigr)
          \|w\|_{L^2} \\
    &\le C\bigl(\|\zeta\|_{L^2}^{1/2}\|\partial_1\zeta\|_{L^2}^{1/2}+\|\zeta\|_{L^2}\bigr)
      \Bigl(\|\partial_1 u_2\|_{L^2}^{1/2}\|\partial_1\partial_2 u_2\|_{L^2}^{1/2}\\
      &\qquad+\|\partial_1 u_2\|_{L^2}+\|\partial_2 u_2\|_{L^2}^{1/2}\|\partial_1\partial_2 u_2\|_{L^2}^{1/2}
          +\|\partial_2 u_2\|_{L^2}\Bigr) \|w\|_{L^2}.
  \end{align*}
   For the typical product
  $\|\zeta\|^{1/2}\|\partial_1\zeta\|^{1/2}
   \times \|\partial_1 u_2\|^{1/2}\|\partial_1\partial_2 u_2\|^{1/2} \|w\|$,
  we write
  \begin{align*}
   &\|\zeta\|^{1/2}\|\partial_1\zeta\|^{1/2}
   \times \|\partial_1 u_2\|^{1/2}\|\partial_1\partial_2 u_2\|^{1/2} \|w\|\\
   &\qquad \le \tfrac12 \|\zeta\|_{L^2}\|\partial_1\zeta\|_{L^2}
    + \tfrac12 \|\partial_1 u_2\|_{L^2}\|\partial_1\partial_2 u_2\|_{L^2} \|w\|_{L^2}^2 .
  \end{align*}
  Repeating this splitting for all analogous products yields
  \begin{align}\label{est_I3}
    |I_3| &\le C_\eta \Bigl(1+\|\partial_1 u_2\|_{L^2}^2+\|\partial_2 u_2\|_{L^2}^2
        +\|\partial_1\partial_2 u_2\|_{L^2}^2\Bigr) \|w\|_{L^2}^2 \notag\\
    &\quad + C\Bigl(\|\zeta\|_{L^2}^2 + \|\zeta\|_{L^2}\|\partial_1\zeta\|_{L^2}
        + \|\partial_1\zeta\|_{L^2}^2 + \|\partial_1\partial_2\zeta\|_{L^2}^2\Bigr).
  \end{align}

  Choosing $\eta$ sufficiently small to absorb the $\|\partial_1 w\|_{L^2}^2$
  terms on the left-hand side of \eqref{eq:energy_w}, we deduce
  \begin{equation}\label{eq:diff_ineq}
    \frac{d}{dt}\|w\|_{L^2}^2 \le a(t) \|w\|_{L^2}^2 + b(t),
  \end{equation}
  where
  \begin{align*}
    a(t) &= C\Bigl(1+\sum_{i=1}^2\bigl(\|\partial_1 u_i\|_{L^2}^2
          +\|\partial_2 u_i\|_{L^2}^2+\|\partial_1\partial_2 u_i\|_{L^2}^2
          +\|u_i\|_{L^2}^2\bigr)\Bigr),\\[2mm]
    b(t) &= C\bigl(\|\zeta\|_{L^2}^2 + \|\partial_1\zeta\|_{L^2}^2
          + \|\partial_2\zeta\|_{L^2}^2 + \|\partial_1\partial_2\zeta\|_{L^2}^2\bigr).
  \end{align*}
  From the a priori estimates for the auxiliary equations, we know that
  \[
    \int_0^T a(s)ds \le C_R, \qquad
    \int_0^T b(s)ds \le C_R \|\zeta\|_G^2 .
  \]
  Applying Gronwall's lemma to \eqref{eq:diff_ineq} (with $w(0)=0$) gives
  \[
    \sup_{0\le t\le T}\|w(t)\|_{L^2}^2 \le e^{\int_0^T a(s)ds} \int_0^T b(s)ds
    \le C_R \|\zeta\|_G^2 .
  \]
  Finally,
  \[
    \|\mathcal F_x(z_1)-\mathcal F_x(z_2)\|_{C([0,T];H)}
    \le \|\zeta\|_{C([0,T];L^2)} + \|w\|_{C([0,T];L^2)}
    \le (C_R+1)\,\|\zeta\|_G,
  \]
  which completes the proof.
\end{proof}

\subsection{Uniform Freidlin--Wentzell LDP}
\label{sec:FW}
We now recall the uniform contraction principle established by Wang \cite{Wang2023},
which does not require exponential tightness.

\begin{theorem}[Uniform contraction principle]\label{th:Wang}
  Let $\Lambda$ be a non-empty set, $Y$ and $Z$ two separable Banach spaces.
  Let $\{\mu^\varepsilon\}_{\varepsilon>0}$ be a family of probability measures on
  $(Y,\mathcal B(Y))$ which satisfies a large deviation principle with good rate
  function $I:Y\to[0,\infty]$.
  For each $\lambda\in\Lambda$, let $T_\lambda : Y\to Z$ be a locally Lipschitz
  mapping in the sense that for every $R>0$ there exists $L_R>0$ such that for all
  $\lambda\in\Lambda$ and all $y_1,y_2\in Y$ with $\|y_1\|_Y,\|y_2\|_Y\le R$,
  \[
    \|T_\lambda(y_1)-T_\lambda(y_2)\|_Z \le L_R \|y_1-y_2\|_Y .
  \]
  Set $\nu_\lambda^\varepsilon = \mu^\varepsilon \circ (T_\lambda)^{-1}$.
  Then the family $\{\nu_\lambda^\varepsilon\}_{\varepsilon>0}$ satisfies a large
  deviation principle on $Z$ uniformly in $\lambda\in\Lambda$, with good rate
  function
  \[
    J_\lambda(z) = \inf \{ I(y) : T_\lambda(y)=z \}, \qquad z\in Z .
  \]
\end{theorem}

In our setting we take
\begin{itemize}
  \item $Y = G = C([0,T];H)\cap L^2(0,T;\widetilde H^{1,1})$,
  \item $Z = C([0,T];H)$,
  \item $\Lambda = \mathcal K$, a non-empty bounded subset of $\widetilde H^{0,1}$,
  \item $T_x = \mathcal F_x$ for $x\in\Lambda$.
\end{itemize}
The LDP for the linear processes $\{z^\varepsilon\}$ in $Y$ has been established
in Theorem~\ref{th:linLDP} (with rate function $I_0$).  Proposition~\ref{prop:Lip}
shows that $\mathcal F_x$ is locally Lipschitz on $Y$, uniformly with respect to
$x$ in any bounded set.  Therefore, by Theorem~\ref{th:Wang} we obtain immediately
the following result.

\begin{theorem}[Freidlin--Wentzell uniform LDP]\label{th:FW}
  For every bounded set $\mathcal K\subset\widetilde H^{0,1}$, the solutions
  $\{u_x^\varepsilon\}$ of \eqref{eq:SNSintro} satisfy the FW-uniform LDP in
  $C([0,T];H)$ with good rate function
  \[
    I_x(u) = \inf\Bigl\{ I_0(z) : \mathcal F_x(z)=u \Bigr\}
    = \inf\Bigl\{ \frac12\int_0^T \|\varphi\|_H^2 dt : u = u_x^\varphi \Bigr\},
  \]
  uniformly for $x\in\mathcal K$.
\end{theorem}

\begin{remark}[Why exponential tightness is not needed]\label{rmk:noET}
  The uniform contraction principle of Cerrai-Paskal
  \cite[Theorem~3.3]{CP} requires the linear process to be \emph{exponentially
  tight} in an auxiliary space $G$, in addition to being locally Lipschitz
  on bounded subsets of $G$ and satisfying an LDP in a
  trajectory space $F\subset G$.  The reason is that the mapping $T$
  that sends the linear orbit to the nonlinear solution is only locally
  Lipschitz from $G$, while the LDP is established in $F$.
  Exponential tightness in $G$ guarantees that, with exponentially small
  probability, the linear trajectories stay inside bounded sets of $G$,
  where the Lipschitz estimate can be safely applied.  Removing
  exponential tightness would leave the possibility that the linear
  process escapes to regions where no Lipschitz control is available,
  destroying the uniform upper bound.

  In this paper, Proposition~\ref{prop:Lip}
  shows that $\mathcal F_x : G \to C([0,T];H)$ is locally Lipschitz
  \emph{on the whole space $G$} (and not only on a subset of it), and
  Theorem~\ref{th:linLDP} provides an LDP for the linear process in $G$
  with a good rate function.  Since any good rate function on a Polish
  space automatically yields exponential tightness, the family
  $\{z^\varepsilon\}$ is exponentially tight in $G$ without any extra
  work.
\end{remark}

\subsection{Uniform Dembo-Zeitouni LDP}
\label{sec:DZ}

In order to obtain the DZ-uniform LDP on compact sets, we need the continuity
of the level sets of the rate function with respect to the initial data.

\begin{lemma}[Hausdorff continuity of the level sets]\label{lem:levelset}
  For any $s\ge 0$, the map $x\mapsto I_x^s = \{u\in C([0,T];H) : I_x(u)\le s\}$
  is continuous with respect to the Hausdorff metric in $C([0,T];H)$, uniformly
  for $x$ in any bounded subset of $\widetilde H^{0,1}$.
\end{lemma}

\begin{proof}
  Let $x,y\in\widetilde H^{0,1}$ with $\|x\|_{\widetilde H^{0,1}},\|y\|_{\widetilde H^{0,1}}\le R$.
  For any $u\in I_x^s$, by definition of the rate function $I_x$, there exists a
  control $\varphi\in L^2(0,T;H)$ such that $u = u_x^\varphi$ and
  $\frac12\int_0^T\|\varphi(t)\|_H^2\,dt \le s$. Here $u_x^\varphi$ denotes the
  unique solution of the deterministic nonlinear controlled (skeleton) equation
  \begin{equation}\label{eq:ske_nonlin}
    \partial_t u_x^\varphi + A u_x^\varphi + \mathbb P\bigl((u_x^\varphi\cdot\nabla)u_x^\varphi\bigr)
    = Q^{1/2}\varphi,\qquad u_x^\varphi(0)=x .
  \end{equation}

  Let $u_y^\varphi$ be the solution of \eqref{eq:ske_nonlin} with initial data $y$
  and the \emph{same} control $\varphi$. Their difference
  $\bar w := u_x^\varphi - u_y^\varphi$ satisfies
  \[
    \partial_t \bar w + A \bar w + \mathbb P\bigl[(u_x^\varphi\cdot\nabla)\bar w
    + (\bar w\cdot\nabla)u_y^\varphi\bigr] = 0,\qquad \bar w(0)=x-y .
  \]
  This equation is structurally identical to the one studied in
  Proposition~\ref{prop:Lip} (with the same $z$ in both $u$'s and without the
  $\zeta$ term). Repeating the energy estimates of that proposition verbatim
  yields the Lipschitz dependence on the initial data:
  \[
    \|u_x^\varphi - u_y^\varphi\|_{C([0,T];H)} \le C(R) \|x-y\|_{L^2},
  \]
  where the constant $C(R)$ depends only on $R$ and on the a priori bounds of
  the skeleton solutions.
  Consequently,
  \[
    \operatorname{dist}_{C([0,T];H)}(u, I_y^s)
    \le \|u_x^\varphi - u_y^\varphi\|_{C([0,T];H)}
    \le C(R) \|x-y\|_{\widetilde H^{0,1}} .
  \]
  By symmetry, the same estimate holds with $x$ and $y$ interchanged. Hence
  \[
    d_H(I_x^s, I_y^s) \le C(R) \|x-y\|_{\widetilde H^{0,1}},
  \]
  which implies the claimed continuity.
\end{proof}

With Lemma~\ref{lem:levelset} and Theorem~\ref{th:FW} in hand, we can apply the
equivalence result between FW- and DZ-uniform LDP on compact parameter sets
established by Salins \cite{Salins}.

\begin{theorem}[Uniform Dembo-Zeitouni LDP]\label{th:mainLDP}
  Let $K$ be any compact subset of $\widetilde H^{0,1}$.
  The family $\{ u_x^\varepsilon \}_{\varepsilon>0}$ satisfies the Dembo--Zeitouni
  uniform large deviation principle in $C([0,T];H)$ with good rate function
  \[
    I(u) = \frac12 \inf\Bigl\{ \int_0^T \|\varphi(s)\|_H^2\,ds :
      u = u_x^\varphi \Bigr\},
  \]
  uniformly for $x\in K$.
\end{theorem}


\section{Failure of Exponential Mixing}\label{sec:mixing}
Due to Lemma~\ref{lem:kernel}, the purely vertical shear modes of the form
$u(x_1,x_2)=(f(x_2),0)$, which constitute the infinite-dimensional kernel
$\mathcal N$ of the horizontal Stokes operator, are completely untouched by
horizontal dissipation. Adjacent fluid layers slide past each other without producing
horizontal velocity gradients, and therefore no horizontal friction acts on
them. This is the fundamental reason why the system cannot forget its initial
conditions exponentially fast, as the following theorem makes precise.

\begin{theorem}[Failure of exponential mixing]\label{th:noExpMix}
  For the anisotropic Navier-Stokes equations \eqref{eq:SNSintro} with additive
  noise, there do not exist constants $C,\gamma>0$ independent of
  $\varepsilon$ such that for all initial data $x,y\in\widetilde H^{0,1}$
  and all $t\ge 0$,
  \begin{align}\label{eq:assumed_mixing}
    \mathbb{E}\|u^\varepsilon(t;x) - u^\varepsilon(t;y)\|_{L^2}^2
    \le C e^{-\gamma t} \|x-y\|_{L^2}^2 .
  \end{align}
\end{theorem}

\begin{proof}
  \noindent\textbf{Step 1. Deterministic obstruction.}
  Take the two initial data
  \[
    x_0 = \begin{pmatrix} \sin x_2 \\ 0 \end{pmatrix}, \qquad y_0 = 0,
  \]
  both belonging to $\widetilde H^{0,1}$.  From Lemma~\ref{lem:kernel} we know
  that $x_0\in\ker A$ and $B(x_0)=0$.  Hence for $\varepsilon=0$,
  \[
    u^0(t;x_0) \equiv x_0,\qquad u^0(t;y_0) \equiv 0,\qquad \forall t\ge 0 .
  \]
  The difference has constant positive norm:
  \[
    \|u^0(t;x_0) - u^0(t;0)\|_{L^2} = \|x_0\|_{L^2} = \sqrt{2}\pi > 0,
    \qquad \forall t\ge 0 .
  \]
  Thus the deterministic equation does not exhibit exponential decay of
  the difference between these two trajectories.

  \medskip
\noindent\textbf{Step 2. Contradiction argument for the stochastic equation.}
First, for the linear process $z^\varepsilon$ of
\eqref{eq:lin}, applying It\^o's formula to $\|z^\varepsilon(t)\|_{L^2}^2$ and
$\|\partial_2 z^\varepsilon(t)\|_{L^2}^2$ exactly as in the proof of
Lemma~\ref{lem:energy} (with the source term $Q^{1/2}\varphi$ replaced by
the stochastic differential $\sqrt\varepsilon Q^{1/2}dW$), and using the
Burkholder-Davis-Gundy inequality, we obtain
the estimate
\[
  \mathbb E\Bigl[ \|z^\varepsilon\|_G^2 \Bigr]
  = \mathbb E\Bigl[ \sup_{0\le t\le T}\|z^\varepsilon(t)\|_{L^2}^2
    + \int_0^T \|z^\varepsilon(t)\|_{\widetilde H^{1,1}}^2\,dt \Bigr]
  \le C\varepsilon\,\mathrm{Tr}(Q) \xrightarrow{\varepsilon\to0} 0 .
\]
Hence $z^\varepsilon\to0$ in $L^2(\Omega;G)$, and therefore also in probability
in $G$.

Now suppose that there exist constants $C,\gamma>0$ independent of
$\varepsilon\in(0,1]$ such that \eqref{eq:assumed_mixing} holds for all
$x,y\in\widetilde H^{0,1}$ and all $t\ge0$.  Now, let $T>0$, by Proposition~\ref{prop:Lip}, the solution map
$\mathcal F_{x_0}: G \to C([0,T];L^2)$ is locally Lipschitz, hence continuous. Because
$z^\varepsilon\to0$ in probability, it implies tightness of the family
$\{\|z^\varepsilon\|_G\}$.
For any $\eta>0$, we can choose $R>0$ such that
$\mathbb{P}(\|z^\varepsilon\|_G > R) \le \eta/2$.  On the event $\{\|z^\varepsilon\|_G \le R\}$, the
Lipschitz estimate gives
\[
\|u^\varepsilon(\cdot;x_0) - \mathcal F_{x_0}(0)\|_{C([0,T];L^2)}
= \|\mathcal F_{x_0}(z^\varepsilon) - \mathcal F_{x_0}(0)\|_{C([0,T];L^2)}
\le L_R \|z^\varepsilon\|_G .
\]
Since the right-hand side converges to $0$
in probability, it shows that
$u^\varepsilon(\cdot;x_0) \to \mathcal F_{x_0}(0)$ in probability in
$C([0,T];L^2)$.  But $\mathcal F_{x_0}(0)$ is precisely the solution of the
deterministic equation with $\varepsilon=0$ and initial data $x_0$.  As
observed in Step~1, $x_0$ is a steady state of the deterministic dynamics,
so $\mathcal F_{x_0}(0) \equiv x_0$.  The same reasoning applied to $y_0=0$
yields $u^\varepsilon(\cdot;0) \to 0$ in probability.
The difference at time $t=T$ satisfies
\[
\|u^\varepsilon(T;x_0)-u^\varepsilon(T;0)\|_{L^2}
\;\xrightarrow{\varepsilon\to0}\;
\|x_0\|_{L^2} \quad\text{in probability}.
\]
Consequently, for any $\delta>0$, there exists \(\varepsilon_0=\varepsilon_0(\delta)\) such that for all
\(\varepsilon\le\varepsilon_0\),
\[
  \mathbb P\Bigl( \bigl| \|u^\varepsilon(T;x_0)-u^\varepsilon(T;0)\|_{L^2} - \|x_0\|_{L^2} \bigr| < \delta \Bigr) > \frac12 .
\]
Taking \(\delta = \frac12\|x_0\|_{L^2}\) and using the reverse triangle inequality, we obtain
\begin{align*}
  \|u^\varepsilon(T;x_0)-u^\varepsilon(T;0)\|_{L^2}
  &\ge \|x_0\|_{L^2} - \bigl| \|u^\varepsilon(T;x_0)-u^\varepsilon(T;0)\|_{L^2} - \|x_0\|_{L^2} \bigr|\\
  &> \|x_0\|_{L^2} - \frac12\|x_0\|_{L^2}\\
  &= \frac12\|x_0\|_{L^2}.
\end{align*}
 Hence,
\[
  \mathbb P\Bigl( \|u^\varepsilon(T;x_0)-u^\varepsilon(T;0)\|_{L^2}
  \ge \frac12\|x_0\|_{L^2} \Bigr)
  \ge\frac12 .
\]
By Markov's inequality,
\[
  \mathbb E\|u^\varepsilon(T;x_0)-u^\varepsilon(T;0)\|_{L^2}^2
  \ge \frac14\|x_0\|_{L^2}^2 \cdot \frac12
  = \frac18\|x_0\|_{L^2}^2 .
\]
On the other hand, the assumed exponential estimate
\eqref{eq:assumed_mixing} would imply
\[
  \mathbb E\|u^\varepsilon(T;x_0)-u^\varepsilon(T;0)\|_{L^2}^2
  \le C e^{-\gamma T} \|x_0\|_{L^2}^2 .
\]
Choosing $T$ sufficiently large so that $C e^{-\gamma T} < \frac18$, we
obtain a contradiction.  Thus no such uniform constants $C,\gamma$ can
exist.
\end{proof}

\begin{remark}[Role of the boundary conditions]\label{rmk:boundary}
  The failure of exponential mixing established in Theorem~\ref{th:noExpMix}
  relies crucially on the periodic boundary conditions.  On a Dirichlet  boundary condition,
  the fixed side walls enforce a no-slip condition, i.e.,
   the fluid velocity is zero at the boundaries.  A purely
  vertical shear flow of the form $(\sin x_2,0)$ would require the fluid to
  slip freely along the lateral walls, which is prohibited.  As a consequence,
  every non-trivial velocity field must necessarily develop horizontal
  variations, thereby activating the horizontal viscous dissipation.
  Mathematically, this is reflected by the anisotropic Poincar\'{e} inequality
  $\|u\|_{L^2}\le C\|\partial_1 u\|_{L^2}$ which holds for every
  divergence-free field vanishing on the boundary.
  With this estimate, the horizontal dissipation provides the
  existence of a unique invariant measure, as proved by
  Sun-Qiu-Tang~\cite{SPL}.
\end{remark}

%

\section{On Invariant Measures}\label{sec:invmeas}
 In Section~\ref{sec:mixing},
 we have shown that the anisotropic
Navier-Stokes equations on the torus cannot be exponentially mixing in any
uniform sense with respect to the noise intensity.
The present section therefore examines what can be said about invariant measures
for the anisotropic system without exponential mixing.

\subsection{Deterministic invariant measures}

We contrast the invariant measures of the deterministic equations
(\(\varepsilon=0\)) for the fully dissipative and the horizontally dissipative
cases.

\begin{proposition}[Deterministic invariant measures]\label{prop:det_inv}
 For the Navier-Stokes equations with \emph{full dissipation}
    \(-\Delta\) on \(\mathbb{T}^{2}\), the unique invariant probability measure
    is \(\delta_{0}\).
 For the equations with \emph{horizontal dissipation only}
    \(-\partial_{1}^{2}\), every Dirac measure \(\delta_{u}\) with
    \(u\in\mathcal{N}=\ker A\) is invariant; in particular there exist
    infinitely many distinct invariant measures.
\end{proposition}

\begin{proof}
  \textbf{Full dissipation.}  The uniqueness of the invariant measure for the
  fully dissipative 2D Navier-Stokes equations on the torus is a classical
  result (see, e.g., \cite{DPZ,LR} and the references therein). We include a
  short proof here merely for comparison with the anisotropic case. Let \(u(t;u_{0})\) be the solution of
  the fully dissipative 2D Navier-Stokes equations with initial data \(u_{0}\).  Taking the \(L^{2}\) inner
  product with \(u\) and using the orthogonality of the nonlinear term yields
  \[
    \frac12\frac{d}{dt}\|u\|_{L^{2}}^{2} + \|\nabla u\|_{L^{2}}^{2}=0 .
  \]
  By the Poincar\'e inequality on the torus,
  there exists \(c>0\) such that \(\|\nabla u\|_{L^{2}}^{2}\ge c\|u\|_{L^{2}}^{2}\).
  Hence
  \[
    \frac{d}{dt}\|u\|_{L^{2}}^{2} + 2c\|u\|_{L^{2}}^{2}\le 0,
  \]
  and Gronwall's lemma gives \(\|u(t)\|_{L^{2}}\le e^{-ct}\|u_{0}\|_{L^{2}}\).
  Consequently every solution converges to zero exponentially fast.
  If \(\mu\) is any invariant probability measure, then for every bounded
  continuous function \(f\) and any \(t>0\)
  \[
    \int_{H} f(u_{0})\,\mu(du_{0})
    = \int_{H} P_{t}f(u_{0})\,\mu(du_{0})
    = \int_{H} \mathbb{E}[f(u(t;u_{0}))]\,\mu(du_{0}) .
  \]
  Letting \(t\to\infty\) and using dominated convergence gives
  \[
    \int_{H} f(u_{0})\,\mu(du_{0}) = f(0),
  \]
  so \(\mu=\delta_{0}\).  The Dirac measure at zero is clearly invariant
  because \(u(t;0)\equiv0\).  Hence \(\delta_{0}\) is the unique invariant
  measure.

  \textbf{Horizontal dissipation only.}  By Lemma~\ref{lem:kernel}, the kernel
  \(\mathcal{N}\) of \(A=-\mathbb{P}\partial_{1}^{2}\) consists of all fields of
  the form \(u_{f}(x_{1},x_{2})=(f(x_{2}),0)\) with \(\int_{\mathbb{T}}f=0\).
  Moreover, for any \(u\in\mathcal{N}\) we have \(B(u)=0\) and \(Au=0\), thus
  \(u(t;u)\equiv u\) is a stationary solution of the deterministic equation
  \eqref{eq:NSintro}.  For each such \(u\) the Dirac measure \(\delta_{u}\) satisfies
  \[
    \int_{H} \mathbb{E}[f(u(t;v))]\,\delta_{u}(dv)
    = f(u(t;u)) = f(u) = \int_{H} f(v)\,\delta_{u}(dv)
  \]
  for all \(t\ge0\) and all bounded measurable \(f\); therefore \(\delta_{u}\)
  is an invariant measure.  Since \(\mathcal{N}\) is an infinite-dimensional
  subspace, there are infinitely many distinct Dirac measures
  \(\{\delta_{u}\}_{u\in\mathcal{N}}\).
\end{proof}

\subsection{A degenerate noise: necessary condition}
\label{sec:degenerate}

We examine a particular additive noise whose range is contained entirely in the kernel
$\mathcal N$.  This example illustrates sharply
the obstacles that the anisotropic dissipation poses to the existence of invariant
measures, and it simultaneously highlights the fundamental role played by the
nonlinear term as the only possible channel of energy exchange between $\mathcal N$
and its orthogonal complement $\mathcal N^\perp$.

\subsubsection*{Construction of the degenerate noise}
Let $\{e_k\}_{k\ge 1}$ be the orthonormal basis of $\mathcal N$ introduced in
Lemma~\ref{lem:kernel}.  Choose a sequence $\{\sigma_k\}_{k\ge 1}\subset(0,\infty)$
with $\sum_{k=1}^{\infty}\sigma_k^2<\infty$, and define a Hilbert--Schmidt operator
$Q^{1/2}:H\to H$ by
\[
  Q^{1/2} u = \sum_{k=1}^{\infty} \sigma_k \langle u, e_k\rangle\, e_k .
\]
Then $\operatorname{Ran}(Q^{1/2})\subset\mathcal N$,
$\operatorname{Tr}_{\mathcal N}(Q)=\sum_{k=1}^{\infty}\sigma_k^2>0$, and
$Q^{1/2}:H\to\widetilde H^{1,1}$ is Hilbert--Schmidt because each $e_k$ lies in
$\widetilde H^{1,1}$.  The noise $\sqrt\varepsilon\,Q^{1/2}W(t)$ therefore acts
only within the vertical shear modes.

\begin{theorem}[Condition on invariant measures]\label{th:degen}
  For the noise constructed above, any invariant probability measure $\mu$ of the
  Markov semigroup associated with \eqref{eq:SNSintro} must satisfy
  $\mu(\mathcal N)=0$.
\end{theorem}

\begin{proof}
  \textbf{Step~1.  Exact solution in the kernel.}
  Let $x\in\mathcal N$.  By Lemma~\ref{lem:kernel}, $A x=0$ and $B(x)=0$, and
  one readily checks that the process
  \[
    u^\varepsilon(t) = x + \sqrt\varepsilon\, Q^{1/2} W(t),\qquad t\ge 0,
  \]
  solves \eqref{eq:SNSintro}.  It remains in $\mathcal N$ for all times and satisfies
  \begin{equation}\label{eq:energy_div}
    \mathbb E\|u^\varepsilon(t)\|_{L^2}^2
    = \|x\|_{L^2}^2 + \varepsilon t\,\operatorname{Tr}_{\mathcal N}(Q)
    \xrightarrow{t\to\infty}\infty .
  \end{equation}

  \textbf{Step~2.  Assume an invariant measure charges $\mathcal N$.}
  Suppose that an invariant probability measure $\mu$ exists with
  $\mu(\mathcal N)>0$.  Define the conditioned probability measure on $\mathcal N$
  by
  \[
    \mu_{\mathcal N}(A) = \frac{\mu(A)}{\mu(\mathcal N)},\qquad
    A\subset\mathcal N\text{ Borel}.
  \]
  Because $\mathcal N$ is a closed invariant subspace
  ($P_t^\varepsilon(x,\mathcal N)=1$ for all $x\in\mathcal N$), the Markov
  semigroup restricted to $\mathcal N$ is well defined.

  Using the invariance of $\mu$, for any $A\subset\mathcal N$ we have
  \begin{align}\label{S5:inv}
    \mu(A) &= \int_{\mathcal N} P_t^\varepsilon(x,A)\,\mu(dx)
            + \int_{H\setminus\mathcal N} P_t^\varepsilon(x,A)\,\mu(dx)\notag \\
           &\ge \int_{\mathcal N} P_t^\varepsilon(x,A)\,\mu(dx).
  \end{align}
  The same inequality holds with $A$ replaced by $\mathcal N\setminus A$.
  Adding the two inequalities and using
  $P_t^\varepsilon(x,A)+P_t^\varepsilon(x,\mathcal N\setminus A)
  =P_t^\varepsilon(x,\mathcal N)=1$, for $x\in\mathcal N$, we have
  \[
    \mu(\mathcal N)
    = \mu(A)+\mu(\mathcal N\setminus A)
    \ge \int_{\mathcal N} 1\,\mu(dx) = \mu(\mathcal N).
  \]
  Hence the inequality in \eqref{S5:inv} must actually be an equality, and consequently
  \[
    \mu(A) = \int_{\mathcal N} P_t^\varepsilon(x,A)\,\mu(dx), \qquad
    \forall A\subset\mathcal N,\; \forall t\ge 0.
  \]
  Dividing by $\mu(\mathcal N)$ we obtain
  \[
    \mu_{\mathcal N}(A) = \int_{\mathcal N} P_t^\varepsilon(x,A)\,\mu_{\mathcal N}(dx),
  \]
  i.e.\ $\mu_{\mathcal N}$ would be an invariant probability measure for the
  semigroup restricted to $\mathcal N$.

  \textbf{Step~3.  Contradiction.}
  However, on $\mathcal N$ the transition probabilities are those of the
  non-degenerate Gaussian semigroup
  $u^\varepsilon(t)=x+\sqrt\varepsilon\,Q^{1/2}W(t)$ with covariance operator
  $\varepsilon t\,Q\big|_{\mathcal N}$.  As is well known, a Gaussian semigroup with linearly
  growing variance on an infinite-dimensional space admits no invariant
  probability measure.  Thus such a $\mu_{\mathcal N}$ cannot exist, contradicting
  the assumption $\mu(\mathcal N)>0$.  Therefore any invariant measure must
  satisfy $\mu(\mathcal N)=0$.
\end{proof}

The degenerate noise constructed above shows that any invariant measure of the
anisotropic Navier-Stokes equations on the torus must be orthogonal to the kernel
$\mathcal N$.  The noise injects energy exclusively into the  vertical
shear modes of $\mathcal N$, whose energy grows linearly in time.  The horizontal
dissipation acts only on $\mathcal N^\perp$, where it would rapidly damp any
incoming excitation.  The nonlinear term $B(u)$ is the sole mechanism that can
transfer energy from $\mathcal N$ into the dissipative region $\mathcal N^\perp$.
If this transfer were efficient enough, it could drain the energy injected into
$\mathcal N$ and deliver it to $\mathcal N^\perp$, where it would be dissipated,
thereby allowing a global statistical balance that gives no mass to $\mathcal N$.

However, in the anisotropic Navier-Stokes equations the energy cascade is
severely constrained by the complete lack of vertical dissipation, and
quantitative estimates on the efficiency of the vertical-to-horizontal transfer
are presently out of reach.  Consequently, the question of whether an invariant
measure actually exists for this degenerate noise-and more generally for the
anisotropic Navier-Stokes equations with additive noise on the torus-is
intimately tied to the analytical difficulty of controlling the
energy exchange between the undamped vertical directions and the dissipative
horizontal directions.

\section{Conclusion and outlook}\label{sec:conclusion}

In this paper, we have established a uniform large deviation principle for
solution paths of the 2D anisotropic Navier--Stokes equations on the torus,
proved the failure of exponential mixing, and obtained necessary conditions
for the existence of invariant measures.

The analysis of the 2D Stokes operator with only horizontal dissipation
provides a clear conceptual framework for understanding the 3D horizontally
dissipative Navier-Stokes equations. The kernel of the dissipation operator
is the fundamental structure that obstructs dissipation, exponential mixing
and statistical balance. Boundary conditions determine the qualitative
behaviour of the system by either eliminating or preserving this kernel.
This perspective extends naturally to three dimensions, where the kernel
again consists of purely vertical shear flows that are immune to horizontal
viscosity.  On periodic domains, the infinite-dimensional kernel  is preserved,
making the long-time dynamics particularly challenging, whereas
no-slip lateral boundaries eliminate the kernel and may restore exponential
mixing.

The pathwise LDP established here does not rely on any long-time statistical
properties of the system.  Extending this approach to three dimensions would require
overcoming the well-known difficulties of global well-posedness and nonlinear
estimates for the 3D anisotropic Navier-Stokes equations, which currently
remain open.  We hope that the structural insights gained in this work will
prove useful for future studies of partially dissipative fluid models,
including the anisotropic MHD and primitive equations.

\section*{Acknowledgments}
The work is supported by the National Natural Science Foundation of China (Nos. 12271293  and  11701269), Natural Science Foundation of Jiangsu Province (No.  BK20231301),  Natural Science Foundation of Shandong Province ( No. ZR2023MA002), the project of Youth Innovation Team of Universities of Shandong Province (No. 2023KJ204). The authors are grateful to the editors and referees for their
very valuable suggestions and comments.

\section*{Data Availability}
There is no data associated with this publication.

\section*{Confict of Interest}
The authors declare no conflict of interest.

\end{document}